\documentclass[a4paper]{amsart}

\usepackage[utf8]{inputenc}
\usepackage[english]{babel}
\usepackage[T1]{fontenc}
\usepackage{lmodern}
\usepackage{dirtytalk}

\usepackage[svgnames]{xcolor}

\usepackage{babelbib}

\usepackage{amsmath, amssymb, amsthm}   
\usepackage{amstext, amsfonts, a4, mathrsfs}
\usepackage{epsfig}
\usepackage{epic, eepic}

\usepackage{comment}
\usepackage{marginnote}

\usepackage[all]{xy} 
\usepackage{centernot}
\usepackage{array}
\usepackage{enumerate}
\usepackage{graphics,graphicx}
\usepackage{float}
\usepackage{ae,aecompl}
\usepackage{eso-pic}
\usepackage{pst-all}
\usepackage{caption}
\usepackage{subcaption}

\usepackage{tabularx}
\usepackage{cases}
\usepackage[shortlabels]{enumitem}

\usepackage{hyperref}

\usepackage{tikz-cd}

\newtheorem{thm}{Theorem}[section]

\newtheorem{lem}[thm]{Lemma}

\newtheorem{coro}[thm]{Corollary}

\newtheorem{mainthm}{Theorem}

\theoremstyle{definition}
\newtheorem{defi}[thm]{Definition}

\newtheorem{example}[thm]{Example}

\theoremstyle{remark}
\newtheorem{rmk}[thm]{Remark}

\numberwithin{equation}{section}
	
\def\N{{\mathbb N}}    
\def\Z{{\mathbb Z}}    
\def\R{{\mathbb R}}    
\def\S{{\mathbb S}}    
\def\T{{\mathbb T}}    

          \def\cC{{\mathcal C}}                      

\usepackage[top=3cm, bottom=3cm, left=3cm, right=3cm]{geometry}

\PassOptionsToPackage{unicode}{hyperref}
\hypersetup{ 
linktocpage=true,
colorlinks=true, 
breaklinks=true, 
pdfencoding=auto,
pdftitle={}, pdfauthor={}, pdfsubject={}
}

\title{An obstruction to fiberwise Anosov flows over 3-dimensional Anosov flows}

\author[N.~Paulet]{Neige Paulet}
 \address{Queen's University, Kingston, Ontario}
\email{neige.paulet@queensu.ca}

\author[D.~Zhang]{Danyu Zhang}
 \address{University of Luxembourg, Luxembourg}
\email{danyu.zhang@uni.lu}

\date{\today}

\begin{document}
\maketitle

\begin{abstract}
We study obstructions preventing a three-dimensional Anosov flow from serving as the base of a fiberwise Anosov flow.
We prove a non-existence result if the base flow admits infinitely many periodic orbits in the same free homotopy class.
We get as a corollary that any $\R$-covered Anosov flow serving as the base of a fiberwise Anosov flow is orbit equivalent to a suspension or a geodesic flow.
\end{abstract}

\setcounter{tocdepth}{1}

\tableofcontents 

\section{Introduction}

Anosov flows have been extensively studied since their introduction in the 1960s, serving as prototype for chaotic dynamic and structural stability.
They are also interesting as they appear naturally in geometry and topology.
Classical examples are suspensions of Anosov diffeomorphisms of tori and the geodesic flows of negatively curved manifolds. 

The question of existence and classification of Anosov flows on a given manifold is central.
We study Anosov flows up to \emph{orbit equivalence}, that is, up to homeomorphisms mapping orbits to orbits.
In his thesis, Tomter \cite{tomterAnosovFlowsInfrahomogeneous1970} classified all homogeneous manifolds admitting Anosov flows: 
they are either nilmanifold bundles over the circle, unit tangent bundles of negatively curved manifolds, 
or mixed spaces obtained by combining these two types. 
We call these the \emph{algebraic} examples, and they provide examples of Anosov flows in higher dimensions, and besides them very few constructions are known. 
The situation is very different in dimension 3. 
Any Anosov flow carried by a manifold which is either a Seifert manifold or a torus bundle over a circle is orbit equivalent to an algebraic Anosov flow (\cite{planteAnosovFlows1972}, \cite{ghysFlotsAnosov3varietes1984}). 
On the other hand, a wide variety of \emph{non-algebraic} Anosov flows have been produced. 
Handel and Thurston \cite{handelAnosovFlowsNew1980} showed that one can glue pieces of geodesic flows to build new non-algebraic flows. 
The surgery of Dehn--Goodman--Fried \cite{goodmanDehnSurgeryAnosov1983}, \cite{friedTransitiveAnosovFlows1983} provided a very powerful tool using Dehn surgery along periodic orbits to construct new Anosov flows, leading to the first examples on hyperbolic 3-manifold.
Bonatti and Langevin \cite{bonattiExempleFlotAnosov1994} constructed the first transitive example transverse to an embedded torus and not orbit equivalent to a suspension. 
Franks and Williams \cite{franksAnomalousAnosovFlows1980} built the first non-transitive Anosov flow, by a blow-up and gluing procedure. 
Later, Béguin--Bonatti--Yu \cite{beguinBuildingAnosovFlows2017} introduced new methods to build Anosov flows by gluing together pieces of flows along a transverse torus boundary, generalized by the first author in \cite{pauletAnosovFlowsDimension2023} for the quasi-transverse torus boundary.

Altogether, dimension three provides a rich zoo of examples, from algebraic models to flows produced by surgeries, blow-ups, and gluings. 
In contrast, in dimension four and above, there is essentially only one class of non-algebraic examples that is known \cite{barthelmeAnomalousAnosovFlows2021}, which are non-transitive.

A promising framework in this direction is that of \emph{fiberwise Anosov flows},
introduced by Farrell-Gogolev \cite{farrell-gogolev}, and further discussed by Barthelmé–Bonatti–Gogolev–Rodriguez-Hertz in \cite{barthelmeAnomalousAnosovFlows2021}. 
Given a $\T^d$-torus bundle $E \to M$, a flow on $E$ is said to be fiberwise Anosov if it projects to a flow on the base $M$ 
and admits a uniform hyperbolic splitting along the $\T^d$ fibers (Definition \ref{def: fiberwise}). 
This construction extends the classical suspension of toral automorphisms, 
and encompass higher-dimensional versions such as the double suspension or Tomter’s example of mixed type. 
A key fact is that whenever the base flow is Anosov, 
a fiberwise Anosov flow automatically yields a new Anosov flow on the total space. 
Thus, fiberwise Anosov flows provide a way to lift three-dimensional dynamics to higher dimensions, 
and a natural candidate for producing new non-algebraic examples. 
In \cite{barthelmeAnomalousAnosovFlows2021}, they provide a dichotomy result for fiberwise Anosov flow: they are either suspension, or have equal dimension of stable and unstable dimension $\geq 3$. This comes from a dichotomy result for Anosov flows admitting pair of periodic orbits freely homotopic.

The central question motivating this work is therefore: 
\emph{which three-dimensional Anosov flows can serve as the base of a fiberwise Anosov flow?} 
Equivalently, can the realm of non-algebraic constructions available in dimension three be transported fiberwise to higher dimensions? 
Or are there intrinsic obstructions that prevent such covering from existing? 

\medskip

In this paper, we obtain the following obstruction. 
We show that if a 3-dimensional Anosov flow admits infinitely many periodic orbits in the same free homotopy class, 
then it cannot be the base of a fiberwise Anosov flow. 

\begin{mainthm} \label{thm: non existence of fiberwise if infinitely homotopic orbit}
    Let $\phi^t$ an Anosov flow on a 3-manifold $M$, such that there is infinitely many periodic orbits in the same free homotopy class.
    Then there is no fiberwise Anosov flow covering $\phi^t$.
\end{mainthm}

The idea of the proof is that along each periodic orbit of the base, the dynamics restricts to an Anosov diffeomorphism of the torus fiber. 
The entropy of these maps grows linearly with the period of the base orbit, 
yet freely homotopic periodic orbits yield conjugate monodromies with the same entropy. 
This contradiction rules out the fiberwise construction. 
As a consequence, all $\R$-covered flows that are not suspensions or geodesic flows are excluded, 
by the classification of Barthelmé–Fenley \cite{barthelmeCountingPeriodicOrbits2017}. 

\begin{coro} \label{coro: fiberwise rigidity R covered}
    If $\Phi^t$ is a fiberwise Anosov flow covering an 3-dimensional $\R$-covered Anosov flow $\varphi^t$, then $\varphi^t$ is orbit equivalent to a suspension or the geodesic flow of a hyperbolic surface.
\end{coro}

This reveals that the class of 3-dimensional Anosov flows that can be lifted fiberwise is limited. 
It includes in particular every generalized Handel–Thurston flow \cite{handelAnosovFlowsNew1980}, 
all flows obtained by gluing pieces of geodesic flows \cite{barbotFlotsAnosovVarietes1996}, 
and a large class of flows arising from Dehn–Goodman–Fried surgeries \cite{fenleyAnosovFlows3manifolds1994}.
Further non $\R$-covered examples with infinite free homotopy classes are found in \cite[Section~5]{barthelmeCountingPeriodicOrbits2017}. 
They are obtained either by blowing up periodic orbits or by the Béguin–Bonatti–Yu gluing construction \cite{beguinBuildingAnosovFlows2017}.
Some are non transitive, some are transitive but non $\R$-covered.
More generally, there exist efficient criteria using features on the orbit space and on the topology of the ambient manifold, that allow to determine whether a given flow admits periodic orbits with infinite free homotopy class.

The question of the existence of non algebraic fiberwise Anosov flows is still open.
At the opposite end of the spectrum, the example of Bonatti and Langevin \cite{bonattiExempleFlotAnosov1994} exhibits a much more rigid free-homotopy behavior: each free homotopy class contains at most one periodic orbit (\cite[Proposition 4.8]{barthelmeCountingPeriodicOrbits2017}), making it a promising candidate for producing non-algebraic fiberwise Anosov flows. 
This is the subject of ongoing work by the authors.

\subsection{Structure of the paper}

In Section~\ref{sec: preli}, we recall preliminary definitions and results on (fiberwise) Anosov flows and entropy.
In Section \ref{sec: infinitely many orbit}, we prove Theorem \ref{thm: non existence of fiberwise if infinitely homotopic orbit}.

\subsection{Acknowledgments}
We are grateful to Christian Bonatti for a discussion that led to the main idea of this paper. 
We thank Thomas Barthelm\'e for many discussions around repeated attempts to construct examples.
We also thank Andrey Gogolev for the suggestion of the project around fiberwise Anosov flows, and many conversations.

\section{Preliminaries}
\label{sec: preli}

\subsection{Fiberwise Anosov flows}

\begin{defi}[Anosov flow]
    A $\cC^1$ flow $\phi^t \colon M \to M$ on a closed manifold $M$ is called an \emph{Anosov flow} if there exist an invariant splitting
    \[
        TM \;=\; E^s \oplus \R X \oplus E^u,
    \]
    where $X$ is the vector field generating $\phi^t$, and constants $C > 0$ and $\lambda > 1$ such that for all $t > 0$
    \begin{align*}
        \|D\phi^{-t} (v^s)\|  &\ge C \lambda^t \|v^s\| \quad \text{for all } v^s \in E^s,\\
        \|D\phi^{t} (v^u)\|   &\ge C \lambda^t \|v^u\| \quad \text{for all } v^u \in E^u,
    \end{align*}
    where $\|\cdot\|$ denotes any Riemannian norm on $M$.
\end{defi}

A \emph{$\T^d$–bundle} $E$ over a base $B$ is a topological space which is locally homeomorphic to a product $U_i \times \T^d$, where $\{U_i\}_{i \in \N}$ is an open cover of $B$.
The total space $E$ is obtained by gluing the charts $U_i \times \T^d$ and $U_j \times \T^d$ along maps
$$
\psi_{ij} \colon U_i \cap U_j \longrightarrow G,
$$
where $G$ is a subgroup of $\mathrm{Homeo}(\T^d)$; the maps $\psi_{ij}(x)$ play the role of transition maps on the fibers above $x \in U_i \cap U_j$.
The group $G$ is called the \emph{structure group} of the bundle $E$.
We refer the reader to \cite{husemoller,steenrod} for background on fiber bundles.

An \emph{affine $\T^d$–bundle} is a torus bundle whose structure group consists only of affine automorphisms of $\T^d$.
Equivalently, it is a locally trivial fiber bundle with structure group
$$
SL(d,\Z) \ltimes \T^d
$$
acting on $\T^d$ by affine transformations $(A,v) x = Ax + v$.

The following definition is due to \cite{barthelmeAnomalousAnosovFlows2021}.
Let $M$ be a closed smooth manifold and let
$$
\T^d \longrightarrow E \overset{\pi}{\longrightarrow} M
$$
be an affine torus bundle.
For $x \in M$, we denote by $\T^d_x$ the fiber over $x$, and by
\[
V E \;=\; \ker D\pi \subset T E
\]
the \emph{vertical subbundle}, consisting of tangent vectors to the fibers of $\pi$.

\begin{defi}[Fiberwise Anosov]\label{def: fiberwise}
    Let $\T^d \to E \overset{\pi}{\to} M$ be an affine torus bundle and let $\phi^t \colon M \to M$ be a flow.
    A flow $\Phi^t \colon E \to E$ is called a \emph{fiberwise Anosov flow over $\phi^t$} if:
    \begin{enumerate}
        \item $\Phi^t$ fibers over $\phi^t$, i.e.\ $\pi \circ \Phi^t = \phi^t \circ \pi$ for all $t$.
        \item There exist a $D\Phi^t$–invariant vertical splitting $V E \;=\; V^s \oplus V^u$,
        a constant $C > 0$ and a constant $\lambda >1$ such that for all $t > 0$,
        \begin{align*}
            \|D\Phi^{-t} v^s\|   &\ge C \lambda^t \|v^s\| \quad \text{for all } v^s \in V^s,\\
            \|D\Phi^{-t} v^u\| &\ge C \lambda^t \|v^u\| \quad \text{for all } v^u \in V^u.
        \end{align*}
    \end{enumerate}
\end{defi}

In this paper we are interested in the case where the base flow $\phi^t$ is a 3–dimensional Anosov flow.
By \cite[Proposition~2.4]{barthelmeAnomalousAnosovFlows2021}, a fiberwise Anosov flow over an Anosov base is itself an Anosov flow on the total space.
Conversely, by \cite[Proposition~2.5]{barthelmeAnomalousAnosovFlows2021}, any Anosov flow $\Phi^t$ on $E$ that fibers over an Anosov flow $\phi^t$ on $M$ is automatically fiberwise Anosov.

We say that a fiberwise Anosov flow $\Phi^t \colon E \to E$ over $\phi^t \colon M \to M$ is \emph{affine} if for each $x \in M$ and $t \in \R$, the map
$$
\Phi^t_x \colon \T^d_x \longrightarrow \T^d_{\phi^t(x)}
$$
corresponding to the restriction of the flow $\Phi^t$ to $\T^d_x$ is an affine diffeomorphism.
This property is independent of the choice of local trivializations at $x$ and $\phi^t(x)$, since the transition maps take values in the affine structure group.

We give two classical examples of fiberwise affine Anosov flows below.

\begin{example} \label{ex: product suspension and mixed}
    \begin{enumerate}
        \item\label{double-susp} \emph{(Double suspension).}
        Let $A \colon \T^m \to \T^m$ and $B \colon \T^n \to \T^n$ be Anosov toral automorphisms.
        Define a flow
        $\tilde\Phi^t \colon \T^m \times \T^n \times \R \longrightarrow \T^m \times \T^n \times \R$
        by $$\tilde\Phi^t(x,y,s) = (x,y,s+t).$$
        Let $\Z$ act on $\T^m \times \T^n \times \R$ by
        \[
        (x,y,s)\cdot 1 = (A^{-1}x, B^{-1}y, s+1).
        \]
        The flow $\tilde\Phi^t$ is $\Z$–equivariant and descends to a flow $\Phi^t$ on the quotient
        $$E := (\T^m \times \T^n \times \R)/\Z.$$
        The induced flow $\Phi^t$ is a fiberwise affine Anosov flow over the suspension of the diffeomorphism $B$ on $\T^n$.

        \item \emph{(Tomter's mixed example).}
        Let $\Sigma$ be a closed hyperbolic surface of genus at least $2$, and let $\phi^t$ denote the geodesic flow on $T^1\Sigma$.
        In \cite{tomterAnosovFlowsInfrahomogeneous1970}, Tomter constructs a fiberwise affine Anosov flow $\Phi^t$ covering $\phi^t$ as follows:
        one finds a representation
        \[
        \rho \colon \pi_1(\Sigma) \longrightarrow SL(4d,\Z)
        \]
        and considers the quotient
        \[
        E := \bigl( T^1 \mathbb{H}^2 \times \T^{4d} \bigr) \big/ \pi_1(\Sigma),
        \]
        where $\pi_1(\Sigma)$ acts diagonally.
        The flow $\Phi^t$ on $E$ is the quotient of the product of the lifted geodesic flow on $T^1\mathbb{H}^2$ with the identity on $\T^{4d}$.
        More details can be found in \cite{barbotAlgebraicAnosovActions2013}.
    \end{enumerate}
\end{example}

For a general $\T^d$-bundle, we denote by
$\rho \colon \pi_1(M) \to \mathrm{Homeo}(\T^d)$
the monodromy representation of the bundle.
If the torus bundle is affine, the image of $\rho$ is contained in the affine group $SL(d,\Z) \rtimes \T^d$.
This representation depends on the choice of base point in $M$, and changing the base point conjugates $\rho$ by an element of $SL(d,\Z)$.

\begin{rmk} \label{rmk: hyperbolic monodromy}
    Let $\Phi^t \colon E \to E$ be a fiberwise Anosov flow covering $\phi^t \colon M \to M$, and let $\rho \colon \pi_1(M) \to \mathrm{Homeo}(\T^d)$ be the monodromy representation of the $\T^d$–bundle.
    For any periodic orbit $\gamma$ of $\phi^t$ in $M$, if $c \in \pi_1(M)$ represents a loop freely homotopic to $\gamma$, then $\rho(c)$ is an Anosov diffeomorphism of $\T^d$.
    Indeed, the restriction of $\Phi^t$ to the subbundle $\pi^{-1}(\gamma)$ over $\gamma$ is a suspension flow, and $\pi^{-1}(\gamma)$ is the mapping torus with monodromy $\rho(c)$.
    The first return map
    \[
    F \colon \T^d_x \longrightarrow \T^d_x,
    \]
    of the fiberwise Anosov flow on the fiber above $x \in \gamma$ is an Anosov diffeomorphism of the torus, conjugated to $\rho(c)$.
    
\end{rmk}

\subsection{Entropy}

We now recall a few facts about entropy for Anosov diffeomorphisms, which will be used in the article.

Let $f \colon \T^d \to \T^d$ be a $\cC^1$ Anosov diffeomorphism.  
It is necessarily transitive, and admits a unique measure of maximal entropy $\mu$, which is ergodic (see, for example, \cite{bowen}).  
Pesin’s entropy formula \cite[Theorem~5.4.5]{pesin} then gives
\begin{equation} \label{eq: entropy formula}
    h_{\mathrm{top}}(f)
    = h_\mu(f)
    = \int_{\T^d} 
      \sum_{\chi_i(x) > 0} k_i(x)\, \chi_i(x)\, d\mu(x),
\end{equation}
where $\chi_i(x)$ are the Lyapunov exponents at $x$, counted with multiplicities $k_i(x)$.  
See \cite[Section~2.1]{pesin} for further details.

If $A \in SL(d,\Z)$ is a linear Anosov automorphism of the torus, its topological entropy is
\[
h_{\mathrm{top}}(A)
= \sum_{|\lambda_i|>1} \log |\lambda_i|,
\]
where $\lambda_i$ are the eigenvalues of $A$.  
For a general Anosov diffeomorphism $f \colon \T^d \to \T^d$, Franks proved that $f$ is conjugate to a linear automorphism $A$ of $\T^d$ \cite{franksAnosovDiffeomorphisms1971}, hence
\[
h_{\mathrm{top}}(f) = h_{\mathrm{top}}(A).
\]

The following lemma shows that for a fiberwise Anosov flow, the entropy of the first return map on the fiber above a periodic orbit of the base grows at least linearly with the period of the orbit.

\begin{lem} \label{lem: entropy linear}
There exists a constant $K > 0$ such that for every periodic orbit $\gamma$ of $\phi^t$ with period $\tau$ and every $x \in \gamma$,  
if $F \colon \T_x \to \T_x$ denotes the first return map of the fiberwise flow along the fiber $\T_x$, then
\[
h_{\mathrm{top}}(F) \;\ge\; K\,\tau.
\]
\end{lem}

\begin{proof}
Let $\chi_i(y)$ denote the Lyapunov exponents of $F$ at a point $y \in \T_x$, with multiplicities $k_i(y)$.  
By the fiberwise Anosov property, for every $v \in V^u$ we have
\[
\|DF(v)\| \;\ge\; C\, \lambda^{\tau} \|v\|,
\]
where $\lambda > 1$ and $C > 0$ are uniform constants.  
Iterating $F$, we obtain
\[
\|DF^k(v)\| \;\ge\; C\, \lambda^{k \tau} \|v\|.
\]
Hence the Lyapunov exponents satisfy
\[
\chi_i(y)
    \;\ge\;
    \limsup_{k \to \infty} 
    \frac{1}{k} \log \bigl( C\, \lambda^{k \tau} \|v\| \bigr)
    = (\log \lambda)\, \tau
\]
Applying Pesin’s formula \eqref{eq: entropy formula}, and possibly renormalizing constants, we obtain
\[
h_\mu(F)
    = \int_{\T_x}
      \sum_{\chi_i(y) > 0}
      k_i(y)\, \chi_i(y)\, d\mu(y)
    \;\ge\; d\, (\log\lambda) \, \tau
\]
where $d$ is the dimension of $V^u$. Since $h_{\mathrm{top}}(F) = h_\mu(F)$ for the measure of maximal entropy, this proves the claim with $K= d (\log\lambda)$.
\end{proof}

\section{Infinitely many freely homotopic periodic orbits}

\label{sec: infinitely many orbit}

In this section, we prove Theorem \ref{thm: non existence of fiberwise if infinitely homotopic orbit}.

\begin{proof}[Proof of Theorem \ref{thm: non existence of fiberwise if infinitely homotopic orbit}]
Consider $M$ a 3-manifold and $\phi^t$ an Anosov flow on $M$.
Suppose that $\phi^t$ admits infinitely many freely homotopic periodic orbits.
The period of the orbits of this infinite family goes to infinity (see for example \cite[Theorem 5]{bowenExpansiveOneparameterFlows1972}).
We take a sequence of periodic points $\{x_n\}$ of $\phi^t$, each of which has period $t_n\geq n$ where $n\in\N_{>0}$, i.e. $\phi^{t_n}(x_n)=x_n$. Let us denote the periodic orbits of $x_n$ by $\gamma_n$.

Suppose there is a fiberwise Anosov flow $\Phi^t$ that fibers over $\phi^t$.
Denote $VE= V^s \oplus V^u$ the vertical hyperbolic splitting from Definition \ref{def: fiberwise}.
Denote the restriction of $\Phi^{t_n}$ to the fiber $\T^d_{x_n}$ by $F_n:\T^d_{x_n}\to \T^d_{x_n}$. Each $F_n$ is an Anosov diffeomorphism of the torus $\T^d \simeq \T^d_{x_n}$.

From the definition of fiberwise Anosov, we have in each fiber $\T^d_{x_n}$ and for any $v\in V^u$ that
$$\|d F_n v\|\geq C\lambda^{t_n}\|v\|\geq C\lambda^{n}\|v\|.$$
We will consider the entropy of the map $F_n$.

We know that every $\gamma_n$ is freely homotopic to $\gamma_0$, for a fixed periodic orbit $\gamma_0$. Let us denote the free homotopy between $\gamma_n$ and $\gamma_0$ by $H_n: [0,1]\times\S^1\to M$. 
Denote $H_n^* E$ the pullback of the $\T^d$-bundle over $[0,1]\times\S^1$.
The monodromy of this subbundle is given by a linear automorphism $A \colon \T^d \to \T^d$. Since both $\gamma_0$ and $\gamma_n$ are periodic orbits of the Anosov flow, $F_n$ is conjugate to the monodromy along $\gamma_n$ and hence is conjugate to $A$. 
They all have the same topological entropy.
It follows from Lemma
\ref{lem: entropy linear} that
$$h_{\mathrm{top}}(A) = h_{\mathrm{top}}(F_n)  \geq K \, t_n \geq K n $$
but the right hand can be arbitrarily large.
\end{proof}

\begin{rmk}
   We remark here that although our original definition of fiberwise Anosov flows requires the bundle to be an affine bundle, we did not need the bundle to be affine for the above proof of Theorem \ref{thm: non existence of fiberwise if infinitely homotopic orbit}. 
Indeed, each $F_n$ is conjugated to the monodromy map $A_n:\T^n \to \T^n$ of the torus bundle over $\gamma_n$ (not necessarily affine), and they are all equal because $\gamma_n$ are freely homotopic.
So they all have the same entropy.
But this is in contradiction with the fact that the entropy of $F_n$ grows as the period of the base orbit $\gamma_n$, hence goes to infinity. 
\end{rmk}

\begin{proof}[Proof of Corollary \ref{coro: fiberwise rigidity R covered}]
We know from {\cite[Theorem B]{barthelmeCountingPeriodicOrbits2017}} that if $\phi^t$ is an $\R$-covered Anosov flow on a closed $3$-manifold $M$ such that every periodic orbit is freely homotopic to at most a finite number of other periodic orbits, then either $\phi^t$ is orbit equivalent to a suspension or, up to finite cover, $\phi^t$ is orbit equivalent to the geodesic flow of a negatively curved surface.
This result together with Theorem \ref{thm: non existence of fiberwise if infinitely homotopic orbit} gives Corollary~\ref{coro: fiberwise rigidity R covered}.
\end{proof}

\begin{rmk} \label{rmk: example infinite many}
One can construct further examples of Anosov flows on 3-manifold satisfying the condition of Theorem \ref{thm: non existence of fiberwise if infinitely homotopic orbit} using the methods of \cite{pauletAnosovFlowsDimension2023} together with the criteria in \cite{barthelmeCountingPeriodicOrbits2017}.
In particular, any Anosov flow containing an \emph{atoroidal piece cut from a skew Anosov flow} (in the sense of \cite{pauletAnosovFlowsDimension2023}) will admit an infinite free homotopy class of periodic orbits; hence it cannot be the base of a fiberwise Anosov flow.
Indeed, the infinite free homotopy class of the original skew flow is contained in the atoroidal piece by \cite[Lemma~2.13]{barthelmeCountingPeriodicOrbits2017} and will persist under the gluing.

Using also the classification of free Seifert pieces\footnote{A free Seifert piece in an Anosov flow is a Seifert piece which is not periodic, that is no periodic orbit is freely homotopic to the fiber for any Seifert fibration} in \cite{barbotFreeSeifertPieces2021}, any Anosov flow whose JSJ decomposition contains a cycle of free Seifert pieces will admit infinitely many freely homotopic periodic orbits, and therefore cannot be the base of a fiberwise Anosov flow.
Such a cycle forms a \emph{skew piece} (morally a piece cut from a skew Anosov flow); as soon as there are more than two pieces, the flow is not a geodesic flow but will exhibit an infinite free homotopy class by the same argument.
This ``skew piece'' viewpoint will be formalized in a work in progress \cite{barthelmeSkewPiecesPseudoAnosov2025}.
\end{rmk}

\begin{rmk}[Perspectives]
The obstruction proved here is given by the presence of infinite free homotopy classes of periodic orbits. 
To construct examples it is natural to look at 3-dimensional Anosov flows lying at the opposite end of the spectrum, i.e. with ``small'' homotopy classes of periodic orbits.
A good candidate is the Bonatti--Langevin flow~\cite{bonattiExempleFlotAnosov1994}, which is not orbit equivalent to a suspension, and where each free homotopy class contains at most one periodic orbit (see~\cite[Proposition~4.8]{barthelmeCountingPeriodicOrbits2017}). 
This makes it a promising base for a non-algebraic fiberwise Anosov flow, and is the purpose of an ongoing work by the authors.
One could also consider more generally totally periodic Anosov flows (\cite{barbotClassificationRigidityTotally2015}) as base candidate, as they have only finitely many periodic orbits in each free homotopy class (\cite[Corollary~4.6]{barthelmeCountingPeriodicOrbits2017}).

Another potential base candidate is the Franks-Williams example~\cite{franksAnomalousAnosovFlows1980}.
In \cite{barthelmeAnomalousAnosovFlows2021}, the authors show that the higher-dimensional construction suggested in \cite{franksAnomalousAnosovFlows1980} cannot be carried out when one restricts to linear gluing maps along the fibers, because the transversality of stable and unstable manifold cannot be achieved.
It would be interesting to investigate whether a similar construction could work with non-linear gluing maps.

\end{rmk}

\bibliographystyle{alpha}
\bibliography{nonexbib.bib}

@book {pesin,
    AUTHOR = {Barreira, Luis and Pesin, Yakov B.},
     TITLE = {Lyapunov exponents and smooth ergodic theory},
    SERIES = {University Lecture Series},
    VOLUME = {23},
 PUBLISHER = {American Mathematical Society, Providence, RI},
      YEAR = {2002},
     PAGES = {xii+151},
      ISBN = {0-8218-2921-1},
   MRCLASS = {37D25 (34D08 37C40 37D10)},
  MRNUMBER = {1862379},
MRREVIEWER = {Viviane\ Baladi},
       DOI = {10.1090/ulect/023},
       URL = {https://doi.org/10.1090/ulect/023},
}

@misc{barthelmeSkewPiecesPseudoAnosov2025,
  title = {Skew Pieces in Pseudo-{{Anosov}} Flows},
  author = {Barthelm{\'e}, Thomas and Lu, Lingfeng and Mann, Kathryn and Paulet, Neige and Salmoiraghi, Federico},
  year = {2025},
  series = {To Appear},
  langid = {english}
}

@book {bowen,
    AUTHOR = {Bowen, Rufus},
     TITLE = {Equilibrium states and the ergodic theory of {A}nosov
              diffeomorphisms},
    SERIES = {Lecture Notes in Mathematics},
    VOLUME = {470},
    EDITOR = {Chazottes, Jean-Ren\'e},
   EDITION = {revised},
      NOTE = {With a preface by David Ruelle},
 PUBLISHER = {Springer-Verlag, Berlin},
      YEAR = {2008},
     PAGES = {viii+75},
      ISBN = {978-3-540-77605-5},
   MRCLASS = {37C40 (28D05 37A25 37D20)},
  MRNUMBER = {2423393},
}

@article{barthelmeCountingPeriodicOrbits2017,
  title = {Counting Periodic Orbits of {{Anosov}} Flows in Free Homotopy Classes},
  author = {Barthelm{\'e}, Thomas and Fenley, S{\'e}rgio R.},
  year = {2017},
  journal = {Commentarii Mathematici Helvetici},
  volume = {92},
  number = {4},
  pages = {641--714},
  issn = {0010-2571},
  doi = {10.4171/CMH/421},
  langid = {english}
}

@article{barbotAlgebraicAnosovActions2013,
  title = {Algebraic {{Anosov}} Actions of Nilpotent {{Lie}} Groups},
  author = {Barbot, Thierry and Maquera, Carlos},
  year = {2013},
  month = jan,
  journal = {Topology and its Applications},
  volume = {160},
  number = {1},
  pages = {199--219},
  issn = {01668641},
  doi = {10.1016/j.topol.2012.10.012},
  urldate = {2023-10-04},
  langid = {english}
}

@article{barbotClassificationRigidityTotally2015,
  title = {Classification and Rigidity of Totally Periodic Pseudo-{{Anosov}} Flows in Graph Manifolds},
  author = {Barbot, Thierry and Fenley, S{\'e}rgio R.},
  year = {2015},
  journal = {Ergodic Theory and Dynamical Systems},
  volume = {35},
  number = {6},
  pages = {1681--1722},
  issn = {0143-3857},
  doi = {10.1017/etds.2014.9},
  langid = {english}
}

@article{barbotFlotsAnosovVarietes1996,
  title = {{Flots d'Anosov sur les vari{\'e}t{\'e}s graph{\'e}es au sens de Waldhausen. (Anosov flows on graph manifolds in the sense of Waldhausen.)}},
  author = {Barbot, Thierry},
  year = {1996},
  journal = {Annales de l'Institut Fourier},
  volume = {46},
  number = {5},
  pages = {1451--1517},
  issn = {0373-0956},
  doi = {10.5802/aif.1556},
  langid = {french}
}

@article{barbotFreeSeifertPieces2021,
  title = {Free {{Seifert}} Pieces of Pseudo-{{Anosov}} Flows},
  author = {Barbot, Thierry and Fenley, S{\'e}rgio R},
  year = {2021},
  month = may,
  journal = {Geometry \& Topology},
  volume = {25},
  number = {3},
  pages = {1331--1440},
  issn = {1364-0380, 1465-3060},
  doi = {10.2140/gt.2021.25.1331},
  urldate = {2024-10-04},
  langid = {english}
}

@article{barthelmeAnomalousAnosovFlows2021,
  title = {Anomalous {{Anosov}} Flows Revisited},
  author = {Barthelm{\'e}, Thomas and Bonatti, Christian and Gogolev, Andrey and Rodriguez Hertz, Federico},
  year = {2021},
  journal = {Proceedings of the London Mathematical Society},
  volume = {122},
  number = {1},
  pages = {93--117},
  issn = {1460-244X},
  doi = {10.1112/plms.12321},
  urldate = {2023-12-06},
  copyright = {{\copyright} 2020 The Authors. The publishing rights in this article are licensed to the London Mathematical Society under an exclusive licence.},
  langid = {english}
}

@article{beguinBuildingAnosovFlows2017,
  title = {Building {{Anosov}} Flows on 3-{{Manifolds}}},
  author = {B{\'e}guin, Fran{\c c}ois and Bonatti, Christian and Yu, Bin},
  year = {2017},
  journal = {Geometry \& Topology},
  volume = {21},
  number = {3},
  pages = {1837--1930},
  publisher = {Mathematical Sciences Publishers}
}

@article{bonattiExempleFlotAnosov1994,
  title = {Un Exemple de Flot d'{{Anosov}} Transitif Transverse {\`a} Un Tore et Non Conjugu{\'e} {\`a} Une Suspension},
  author = {Bonatti, Christian and Langevin, R{\'e}mi},
  year = {1994},
  journal = {Ergodic Theory and Dynamical Systems},
  volume = {14},
  number = {4},
  pages = {633--643},
  publisher = {Cambridge University Press}
}

@article{bowenExpansiveOneparameterFlows1972,
  title = {Expansive One-Parameter Flows},
  author = {Bowen, Rufus and Walters, Peter},
  year = {1972},
  journal = {Journal of Differential Equations},
  volume = {12},
  pages = {180--193},
  issn = {0022-0396},
  doi = {10.1016/0022-0396(72)90013-7},
  langid = {english}
}

@article{fenleyAnosovFlows3manifolds1994,
  title = {Anosov Flows in 3-Manifolds},
  author = {Fenley, S{\'e}rgio R.},
  year = {1994},
  journal = {Annals of Mathematics. Second Series},
  volume = {139},
  number = {1},
  pages = {79--115},
  issn = {0003-486X},
  doi = {10.2307/2946628},
  langid = {english}
}

@incollection{franksAnomalousAnosovFlows1980,
  title = {Anomalous {{Anosov}} Flows},
  booktitle = {Global Theory of Dynamical Systems},
  author = {Franks, John and Williams, Bob},
  year = {1980},
  pages = {158--174},
  publisher = {Springer}
}

@inproceedings{franksAnosovDiffeomorphisms1971,
  title = {Anosov Diffeomorphisms},
  booktitle = {Proceedings of the {{Symposium}} on {{Differential Equations}} and {{Dynamical Systems}}},
  author = {Franks, J.},
  editor = {Chillingworth, David},
  year = {1971},
  pages = {142--143},
  publisher = {Springer Berlin Heidelberg},
  address = {Berlin, Heidelberg},
  isbn = {978-3-540-36662-1}
}

@article{friedTransitiveAnosovFlows1983,
  title = {Transitive {{Anosov}} Flows and Pseudo-{{Anosov}} Maps},
  author = {Fried, David},
  year = {1983},
  journal = {Topology},
  volume = {22},
  number = {3},
  pages = {299--303},
  publisher = {Elsevier}
}

@article{ghysFlotsAnosov3varietes1984,
  title = {Flots d'{{Anosov}} Sur Les 3-Vari{\'e}t{\'e}s Fibr{\'e}es En Cercles},
  author = {Ghys, Etienne},
  year = {1984},
  month = mar,
  journal = {Ergodic Theory and Dynamical Systems},
  volume = {4},
  number = {1},
  pages = {67--80},
  issn = {0143-3857, 1469-4417},
  doi = {10.1017/S0143385700002273},
  urldate = {2024-03-21},
  langid = {english}
}

@article{goodmanDehnSurgeryAnosov1983,
  title = {Dehn Surgery on {{Anosov}} Flows},
  author = {Goodman, Sue},
  editor = {Palis, J.},
  year = {1983},
  journal = {Geometric Dynamics},
  pages = {300--307},
  publisher = {Springer Berlin Heidelberg},
  address = {Berlin, Heidelberg},
  isbn = {978-3-540-40969-4}
}

@article{handelAnosovFlowsNew1980,
  title = {Anosov Flows on New 3-{{Manifolds}}},
  author = {Handel, Michael and Thurston, William P.},
  year = {1980},
  month = jun,
  journal = {Inventiones mathematicae},
  volume = {59},
  number = {2},
  pages = {95--103},
  issn = {1432-1297},
  doi = {10.1007/BF01390039}
}

@misc{pauletAnosovFlowsDimension2023,
  title = {Anosov Flows in Dimension 3 from Gluing Building Blocks with Quasi-Transverse Boundary},
  author = {Paulet, Neige},
  year = {2023},
  month = dec,
  number = {arXiv:2312.13054},
  eprint = {2312.13054},
  primaryclass = {math},
  publisher = {arXiv},
  urldate = {2023-12-21},
  archiveprefix = {arXiv},
  copyright = {Licence Creative Commons Attribution - Pas d'utilisation commerciale - Pas de modification 4.0 International (CC-BY-NC-ND)}
}

@article{planteAnosovFlows1972,
  title = {Anosov {{Flows}}},
  author = {Plante, Joseph F.},
  year = {1972},
  journal = {American Journal of Mathematics},
  volume = {94},
  number = {3},
  eprint = {2373755},
  eprinttype = {jstor},
  pages = {729--754},
  publisher = {Johns Hopkins University Press},
  issn = {0002-9327},
  doi = {10.2307/2373755},
  urldate = {2024-03-21}
}

@inproceedings{tomterAnosovFlowsInfrahomogeneous1970,
  title = {Anosov Flows on Infra-Homogeneous Spaces},
  booktitle = {Global {{Analysis}} ({{Proc}}. {{Sympos}}. {{Pure Math}}., {{Vol}}. {{XIV}}, {{Berkeley}}, {{Calif}}., 1968)},
  author = {Tomter, Per},
  year = {1970},
  pages = {299--327}
}

@book {husemoller,
    AUTHOR = {Husemoller, Dale},
     TITLE = {Fibre bundles},
    SERIES = {Graduate Texts in Mathematics},
    VOLUME = {20},
   EDITION = {Third},
 PUBLISHER = {Springer-Verlag, New York},
      YEAR = {1994},
     PAGES = {xx+353},
      ISBN = {0-387-94087-1},
   MRCLASS = {55-01 (19Lxx 55-02 55Rxx 57R20 57R22)},
  MRNUMBER = {1249482},
MRREVIEWER = {Jo\v ze\ Vrabec},
       DOI = {10.1007/978-1-4757-2261-1},
       URL = {https://doi.org/10.1007/978-1-4757-2261-1},
}

@book {steenrod,
    AUTHOR = {Steenrod, Norman},
     TITLE = {The topology of fibre bundles},
    SERIES = {Princeton Landmarks in Mathematics},
      NOTE = {Reprint of the 1957 edition,
              Princeton Paperbacks},
 PUBLISHER = {Princeton University Press, Princeton, NJ},
      YEAR = {1999},
     PAGES = {viii+229},
      ISBN = {0-691-00548-6},
   MRCLASS = {55-02 (01A75 55Rxx)},
  MRNUMBER = {1688579},
}

@article {farrell-gogolev,
    AUTHOR = {Farrell, F. Thomas and Gogolev, Andrey},
     TITLE = {On bundles that admit fiberwise hyperbolic dynamics},
   JOURNAL = {Math. Ann.},
  FJOURNAL = {Mathematische Annalen},
    VOLUME = {364},
      YEAR = {2016},
    NUMBER = {1-2},
     PAGES = {401--438},
      ISSN = {0025-5831,1432-1807},
   MRCLASS = {37D20 (53C21)},
  MRNUMBER = {3451392},
MRREVIEWER = {Thomas\ Barthelm\'e},
       DOI = {10.1007/s00208-015-1218-8},
       URL = {https://doi.org/10.1007/s00208-015-1218-8},
}

\end{document}